\documentclass[11pt,letterpaper,reqno,oneside]{amsart}
\usepackage[letterpaper, includehead, includefoot,top=30mm, bottom=30mm, left=30mm, right=30mm,marginparwidth=25mm, marginparsep=4mm]{geometry}
\usepackage{times}
\usepackage{comment}
\usepackage{amsfonts}
\usepackage{amsmath}
\usepackage{amssymb}
\usepackage{amsthm}
\usepackage{thmtools}
\usepackage[T1]{fontenc}
\usepackage{enumerate}
\usepackage{booktabs}
\usepackage{caption}
\usepackage{graphicx}

\usepackage{color}
\usepackage{hyperref}
\hypersetup{
    colorlinks=true,
    linkcolor=blue,
    citecolor=blue,
    filecolor=blue,
    urlcolor=blue
}

\usepackage[nameinlink,capitalize,noabbrev]{cleveref}

\usepackage{calrsfs}

\usepackage{tikz}
\usetikzlibrary{calc,intersections}
\usepackage{pgfplots}
\usepgfplotslibrary{patchplots}
\usetikzlibrary{patterns, positioning, arrows}
\pgfplotsset{compat=1.15}

\usepackage{todonotes}

\declaretheorem[name=Theorem,numberwithin=section]{theorem}
\declaretheorem[sibling=theorem,name=Corollary]{corollary}
\declaretheorem[sibling=theorem,name=Lemma]{lemma}
\declaretheorem[sibling=theorem,name=Proposition]{proposition}

\declaretheorem[sibling=theorem,name=Remark,style=remark]{remark}
\declaretheorem[sibling=theorem,name=Definition,style=definition]{definition}

\declaretheorem[sibling=theorem,name=Example,style=definition]{example}

\newtheorem{introthm}{Theorem}

\newcommand\C{\mathbf{C}}

\newcommand{\PP}[0]{\ensuremath{\mathbf{P}}}
\newcommand{\CC}[0]{\ensuremath{\mathbf{C}}}
\newcommand{\ZZ}[0]{\ensuremath{\mathbf{Z}}}

\newcommand{\tvarphi}[0]{\ensuremath{\widetilde{\varphi}}}

\newcommand{\Aut}[0]{\ensuremath{\operatorname{Aut}}}
\newcommand{\AutL}[0]{\ensuremath{\operatorname{Aut}_L}}

\newcommand{\diag}[0]{\ensuremath{\operatorname{diag}}}

\newcommand{\spar}[0]{\ensuremath{\operatorname{Spar}}}
\newcommand{\vars}[0]{\ensuremath{\operatorname{Vars}}}
\newcommand{\PGT}[0]{\ensuremath{\operatorname{PGT}}}
\newcommand{\PGP}[0]{\ensuremath{\operatorname{PGP}}}

\newcommand{\PGL}[0]{\ensuremath{\operatorname{PGL}}}
\newcommand{\GL}[0]{\ensuremath{\operatorname{GL}}}

\newcommand{\IX}{\ensuremath{I_X}}



\begin{document}

\title{Automorphism Groups of Rigid Complete Intersections}

\author[J. Duque Franco]{Jorge Duque Franco}
\address{Dirección de Investigaci\'on, Vicerrector\'ia Acad\'emica and Instituto de Matem\'aticas, Universidad de Talca,
  Casilla 721, Talca, Chile}
\email{jorge.duque@utalca.cl}

\author[A. Liendo]{Alvaro Liendo}
\address{Instituto de Matem\'aticas, Universidad de Talca,
  Casilla 721, Talca, Chile}
\email{aliendo@utalca.cl}

\author[A.~J. Palomino]{Ana Julisa Palomino}
\address{Instituto de Matem\'aticas, Universidad de Talca,
  Casilla 721, Talca, Chile}
\email{ana.palomino@utalca.cl}

\date{\today}

\thanks{
  \mbox{\hspace{11pt}}
  \mbox{\hspace{11pt}}The three authors were partially supported by Fondecyt Projects 1240101 and 13250049. The third author was also partially supported by CONICYT-PFCHA/Doctorado Nacional/folio 21240560}

\begin{abstract}
We study the automorphism groups of complete intersections of hypersurfaces of strictly increasing degrees in projective space. Under a combinatorial rigidity condition on the tuple of defining polynomials, we show that every automorphism of the complete intersection extends to an automorphism of each defining hypersurface, so that its automorphism group is the intersection of the automorphism groups of the defining hypersurfaces. We apply this principle to two natural families of complete intersections of two hypersurfaces of different degrees. For complete intersections of two Fermat hypersurfaces, we determine the automorphism group in every smooth case. For complete intersections of a Klein hypersurface with the reverse-order Klein hypersurface, we describe the automorphism group under an explicit arithmetic condition relating the two degrees, with Klein hypersurfaces of Wagstaff type as a natural source of examples.
\end{abstract}
\maketitle

\section*{Introduction}

Let $X \subseteq \PP^{n+r}$ be a complete intersection of dimension $n \geq 1$ and codimension $r \geq 1$. The study of its automorphism group $\Aut(X)$ rests on two fundamental properties: the extension of regular automorphisms to linear automorphisms of the ambient projective space (so that $\Aut(X)$ is a finite subgroup of $\PGL_{n+r+1}(\C)$), and the triviality of the group in the generic case. For smooth hypersurfaces ($r = 1$) of degree $d \geq 3$ with $(n,d) \neq (1,3),(2,4)$, these properties were classically established by Matsumura--Monsky and Chang \cite{MM64,Cha78}; see \cite{Ko19} for a modern account. In the higher-codimension setting ($r \geq 2$), Benoist \cite{Ben13} proved that smooth complete intersections of dimension at least two (with a few exceptions) possess finite automorphism groups that extend to the ambient space, while Chen, Pan, and Zhang \cite{CPZ24} recently prove that the automorphism group of a general complete intersection is trivial. Despite these generic triviality results, explicitly determining the automorphism groups for specific, highly symmetric varieties remains a central problem. While this has been extensively explored in the hypersurface case, where several recent works have studied explicit families and constrained their possible structure \cite{GL11,GL13,OY19,WY20,LZ22,Zhe22,GALM22,Ess24,LP25}, the explicit construction and description of automorphism groups for highly symmetric families of higher-codimension complete intersections remains a largely open challenge.

Two classical families with unusually large automorphism groups have attracted sustained attention in this context: the \emph{Fermat hypersurface} of degree $d$, defined by
\[
F \;=\; x_0^d + x_1^d + \cdots + x_{n+1}^d \;=\; 0,
\]
and the \emph{Klein hypersurface} of degree $d$, defined by
\[
K \;=\; x_0^{d-1} x_1 + x_1^{d-1} x_2 + \cdots + x_n^{d-1} x_{n+1} + x_{n+1}^{d-1} x_0 \;=\; 0.
\]
The automorphism group of the Fermat hypersurface is the semidirect product $(\ZZ/d\ZZ)^{n+1} \rtimes \mathfrak{S}_{n+2}$ \cite{Kon02}; in fact, Esser--Li \cite{EL25} and, independently, Yang--Yu--Zhu \cite{YYZ25} have shown that the Fermat hypersurface realizes the maximum order of $\Aut(X)$ among smooth hypersurfaces of fixed dimension and degree, with a finite list of exceptions.  The Klein case has a long history, going back to Klein's study of the quartic plane curve \cite{Klein} and continuing through Adler \cite{Adl78}, Dolgachev \cite{Dol12}, Harui \cite{Har19}, and Oguiso--Yu \cite{OY19}; aside from a handful of sporadic low-dimensional cases, the recent paper \cite{GALMVL24} establishes
\[
\Aut(V(K)) \;\cong\; (\ZZ/m\ZZ) \rtimes (\ZZ/(n+2)\ZZ),
\qquad
m \;=\; \frac{(d-1)^{n+2} - (-1)^{n+2}}{d}.
\]
The Klein hypersurface acquires further significance from \cite[Thm.~3.7]{GL13}: every smooth hypersurface of dimension $n$ and degree $d$ admitting an automorphism of prime order $p > (d-1)^n$ is isomorphic to a Klein hypersurface, $n+2$ is prime, and $p$ equals the \emph{generalized Wagstaff prime}
\[
p \;=\; \frac{(d-1)^{n+2} + 1}{d}.
\]
The case $d = 3$ recovers the classical Wagstaff primes and $d = -1$ the Mersenne primes; in both cases there are conjecturally infinitely many \cite{dubner2000primes,BLS75}.  A Klein hypersurface realizing a generalized Wagstaff prime as the order of an automorphism is said to be of \emph{Wagstaff type} \cite{GALMVL24}.

\medskip

    In this paper we introduce and study complete intersection analogues of these families.  From this point onward we adopt the following conventions: the ambient projective space is $\PP^{n+r}$, so that hypersurfaces have dimension $n+r-1$ and complete intersections of $r$ hypersurfaces have dimension $n$.  Also, unless otherwise stated, we assume $r \geq 2$.  Concretely, we work over $\C$ and let $V$ be a vector space of dimension $n+r+1$ with fixed basis $\beta = \{e_0, \ldots, e_{n+r}\}$ and dual basis $\beta^* = \{x_0, \ldots, x_{n+r}\}$.  By a \emph{complete intersection} of multidegree $(d_1, \ldots, d_r)$ in $\PP^{n+r} = \PP(V)$, we mean a closed subscheme $X = V(F_1) \cap \cdots \cap V(F_r)$ of dimension $n$, defined by a regular sequence of homogeneous polynomials $F_1, \ldots, F_r$ in $S(V^*)$ where each $F_j$ has degree $d_j$.  By a \emph{smooth complete intersection} of multidegree $(d_1, \ldots, d_r)$ in $\PP^{n+r} = \PP(V)$ we mean a smooth subvariety $X = V(F_1) \cap \cdots \cap V(F_r)$, where $F_j$ is a homogeneous polynomial of degree $d_j$.  Such a variety has dimension $n$.

Our starting point is the elementary observation that, when the degrees are strictly increasing $d_1 < d_2 < \cdots < d_r$, the graded component $(\IX)_{d_1}$ of the homogeneous ideal of $X$ is one-dimensional, so every automorphism of $X$ preserves $V(F_1)$ up to a scalar (\cref{lem:F1-preserved}).  To upgrade this to $\Aut(X)$ we introduce a combinatorial condition on the tuple $(F_1, \ldots, F_r)$ which we call \emph{rigidity} (\cref{def:rigidity}): for each $j \geq 2$, no nonzero element of $\sum_{i<j} F_i \cdot S^{d_j-d_i}(V^*)$ has all its monomials in the symmetric-group orbit $\mathcal{O}(F_j)$ of the support of $F_j$.  Combined with the results of \cite{GALMVL24} on the shape of automorphisms of $V(F_1)$, which constrain automorphisms via combinatorial conditions on the monomial support of $F_1$, namely the \emph{sparsity} $\spar(F_1)$ and a partial order $\leq_{F_1}$ on the variables (\cref{def:spar-order}), rigidity yields our first main result.

\begin{introthm}
\label{introthm:CI-main}
Let $X = V(F_1) \cap \cdots \cap V(F_r) \subseteq \PP^{n+r}$ be a complete intersection of dimension $n \geq 1$ and multidegree $(d_1, \ldots, d_r)$ with $3 \leq d_1 < d_2 < \cdots < d_r$ and $(\dim V(F_1), d_1) \neq (2, 4)$.  Assume either $n \geq 2$, or $n = 1$, $r = 2$ and $X$ smooth.  Suppose $V(F_1)$ is smooth, $\spar(F_1) > 4$, the poset $(\beta^*, \leq_{F_1})$ is trivial, and the tuple $(F_1, \ldots, F_r)$ is rigid. Then
\[
\Aut(X) \;=\; \bigcap_{j=1}^{r} \Aut(V(F_j))
\]
as subgroups of $\PGL(V)$. Moreover, if each $V(F_j)$ is smooth and $d_j \neq 4$ whenever $n = 1$, then every automorphism of $X$ of prime power order $q = p^s$ with $p \nmid d_1 \cdots d_r$ satisfies $(1 - d_j)^{\ell_j} \equiv 1 \pmod{q}$ for some $\ell_j \in \{1, \ldots, n+r+1\}$, for every $j \in \{1, \ldots, r\}$.
\end{introthm}

We apply \cref{introthm:CI-main} to two natural families in codimension two. In both families considered below, the defining hypersurfaces are smooth and hence irreducible. Since their degrees are different, their defining polynomials are nonassociate and therefore coprime; consequently, they form a regular sequence, and their intersection is a complete intersection. The first family consists of complete intersections of two Fermat hypersurfaces in $\PP^{n+2}$. Rigidity of the pair $(F_1, F_2)$ is automatic, and \cref{introthm:CI-main} applies.

\begin{introthm}
\label{introthm:fermat}
Let $X = V(F_1) \cap V(F_2) \subseteq \PP^{n+2}$ be a complete intersection of two Fermat hypersurfaces of degrees $3 \leq d_1 < d_2$ with $n \geq 1$ and $(n, d_j) \neq (1, 4)$ for $j = 1, 2$. Assume either $n \geq 2$, or $n = 1$ and $X$ is smooth. Then$$\Aut(X) \;\cong\; (\ZZ/\gcd(d_1, d_2)\ZZ)^{n+2} \rtimes \mathfrak{S}_{n+3},$$where $\mathfrak{S}_{n+3}$ acts by permutation of the variables and the abelian factor by diagonal multiplication by appropriate roots of unity.
\end{introthm}

The smoothness of such complete intersections depends on the arithmetic of their degrees. We establish necessary and sufficient conditions for a Fermat complete intersection to be smooth in \cref{prop:fermat-CI-smooth}.

The Klein analogue is more delicate.  The cyclic symmetry of the Klein polynomial $K$ admits a natural variant of degree $d'$ obtained by reversing the cyclic order of the variables, the \emph{reverse-order Klein polynomial} 
\[
K' \;=\; x_1^{d'-1} x_0 + x_2^{d'-1} x_1 + \cdots + x_{n+2}^{d'-1} x_{n+1} + x_0^{d'-1} x_{n+2}.
\]
By a \emph{Klein complete intersection} of multidegree $(d_1, d_2)$ we mean $X = V(K_1) \cap V(K_2) \subseteq \PP^{n+2}$, where $K_1$ is the Klein polynomial of degree $d_1$ and $K_2$ the reverse-order Klein polynomial of degree $d_2$, with $d_1 < d_2$.  The interaction between $V(K_1)$ and $V(K_2)$ connects this family directly to Klein hypersurfaces of Wagstaff type.  We show in \cref{prop:reverse} that, whenever $d$ and $d'$ satisfy
\[
(d-1)(d'-1) \equiv 1 \pmod{m},
\qquad
m \;=\; \frac{(d-1)^{n+3} - (-1)^{n+3}}{d},
\]
the diagonal automorphism of order $m$ of the Klein hypersurface $V(K)$ also preserves $V(K')$. This leads to our last result.  We also record some necessary arithmetic conditions for the smoothness of $X$ in \cref{cor:smooth-necessary}.

\begin{introthm}\label{introthm:klein}Let $X = V(K_1) \cap V(K_2) \subseteq \PP^{n+2}$ be a Klein complete intersection of dimension $n \geq 2$ and multidegree $(d_1,d_2)$ with $3 \leq d_1 < d_2$ and $(n, d_1) \neq (2, 3)$. If$$(d_1 - 1)(d_2 - 1) \equiv 1 \pmod{m},
\qquad
m \;=\; \frac{(d_1 - 1)^{n+3} - (-1)^{n+3}}{d_1},$$then$$\Aut(X) \;\cong\; \Aut(V(K_1)) \;\cong\; (\ZZ/m\ZZ) \rtimes (\ZZ/(n+3)\ZZ).$$
\end{introthm}

\medskip

The paper is organized as follows.   \Cref{sec:prelim} fixes notation and recalls the differential method \cite{OY19,GALMVL24} and the order criterion \cite{GL13,GALM22}.  \Cref{sec:aut-CI} develops the main theory of this paper and proves \cref{introthm:CI-main}.  \Cref{sec:fermat} treats Fermat complete intersections and proves \cref{introthm:fermat}.  \Cref{sec:klein} treats Klein complete intersections and proves \cref{introthm:klein}.

\section{Preliminaries}
\label{sec:prelim}

In this section we fix notation and recall some results on automorphisms of smooth hypersurfaces that we will need throughout the paper. 

We work over the field of complex numbers $\C$, and all varieties are projective unless otherwise stated. Throughout the paper, we fix integers $n \geq 1$ and $r \geq 2$, and let $V$ denote a complex vector space of dimension $n+r+1$ with a fixed basis $\beta = \{e_0, \dots, e_{n+r}\}$ and corresponding dual basis $\beta^* = \{x_0, \dots, x_{n+r}\}$ of $V^*$.  We write $\PP(V) = \PP^{n+r}$ for the corresponding projective space.  The choice of $\beta$ induces an isomorphism between the symmetric algebra $S(V^*)$ and the polynomial ring $\CC[x_0, \dots, x_{n+r}]$.  With this convention, hypersurfaces $V(F) \subset \PP^{n+r}$ have dimension $n+r-1$, and complete intersections $V(F_1) \cap \cdots \cap V(F_r) \subset \PP^{n+r}$ of $r$ hypersurfaces have dimension $n$.  The case $r = 1$ recovers the convention of a single hypersurface in $\PP^{n+1}$ of dimension $n$.

The group of automorphisms of $\PP^{n+r} = \PP(V)$ is the projective linear group $\PGL(V)$, obtained as the quotient of $\GL(V)$ by its center; we write $\pi \colon \GL(V) \to \PGL(V)$ for the canonical projection.  For an element $\varphi \in \PGL(V)$, any preimage $\tvarphi \in \GL(V)$ under $\pi$ is called a \emph{lift} of $\varphi$; lifts are unique up to scalar.  For a homogeneous polynomial $F \in S(V^*)$ and a lift $\tvarphi$, the action $\tvarphi^*(F) = F \circ \tvarphi$ is well defined, while the condition $\tvarphi^*(F) = a F$ for some $a \in \CC^*$ depends only on $\varphi$.

For a projective variety $X \subset \PP^{n+r}$ we denote by $\Aut(X)$ its group of automorphisms, and by $\AutL(X) \subseteq \Aut(X)$ the subgroup of those automorphisms that extend to an automorphism of the ambient $\PP^{n+r}$. By a classical result of Matsumura and Monsky \cite{MM64}, $\Aut(X) = \AutL(X)$ whenever $X$ is a smooth hypersurface of dimension at least one and degree $d \geq 3$, with the exception of smooth plane cubic curves and smooth quartic surfaces in $\PP^3$. For complete intersections, the situation is analogous. Benoist \cite[Th\'eor\`eme~3.1]{Ben13} established this property for smooth complete intersections of dimension $n \geq 2$, with the exceptions of smooth quadrics, smooth intersections of two quadrics, and smooth intersections of a quadric and a cubic in $\PP^4$. The smoothness assumption is not needed for the extension itself: as noted by Kontogeorgis \cite[\S 2]{Kon02}, the extension behavior holds for geometric complete intersections of higher dimension without requiring smoothness. For smooth curves in $\PP^3$ ($n = 1$, $r = 2$), the property holds as a consequence of the results of Ciliberto and Lazarsfeld on complete linear series \cite[Corollary~2.5 and Theorem~2.6]{CL84}. The precise version needed in this paper is stated in the following proposition, and a proof is provided for the sake of completeness.

\begin{proposition}
\label{prop:AutL}
Let $X = V(F_1) \cap \cdots \cap V(F_r) \subset \PP^{n+r}$ be a complete intersection
of dimension $n \geq 1$ and codimension $r \geq 2$, with multidegree
$(d_1, \ldots, d_r)$ satisfying $3 \leq d_1 < d_2 < \cdots < d_r$. Assume either
$n \geq 2$, or $n = 1$, $r = 2$ and $X$ smooth. Then
\[
\Aut(X) = \AutL(X).
\]
\end{proposition}

\begin{proof}
Let $\varphi \in \Aut(X)$ be an automorphism. Since $X \subset \PP^{n+r}$ is a complete intersection, using the Koszul complex we obtain that the natural restriction map of global sections is an isomorphism
\begin{equation}
\label{eq:rest}
H^0(\PP^{n+r}, \mathcal{O}_{\PP^{n+r}}(1)) \xrightarrow{\sim} H^0(X, \mathcal{O}_X(1)).
\end{equation}
Via this isomorphism, the existence of a projective lift of $\varphi$ in $\PGL(V)$ reduces to showing that $\varphi^*\mathcal{O}_X(1) \cong \mathcal{O}_X(1)$.

Assume first $n \geq 3$. By the Grothendieck-Lefschetz theorem, $\operatorname{Pic}(X) \cong \ZZ \cdot \mathcal{O}_X(1)$ \cite[Exposé XII, Corollaire 3.7]{SGA2}, so every automorphism preserves the unique ample generator, and $\varphi^*\mathcal{O}_X(1) \cong \mathcal{O}_X(1)$. 

Now assume that $n = 2$. By \cite[Corollary~7.2.3]{CS24}, $\operatorname{Pic}(X)$ is torsion-free. By the adjunction formula, the dualizing sheaf is invertible and is given by $\omega^\circ_X \cong \mathcal{O}_X(k)$ with $k = \sum_{i=1}^r d_i - (r+3)$ \cite[~III Thm 7.11]{Har77}. Since $r \geq 2$ and $d_1 \geq 3$, $d_2 \geq 4$, we have $k > 0$. Every automorphism preserves $\omega^\circ_X$, so $\varphi^*\mathcal{O}_X(k) \cong \mathcal{O}_X(k)$, and torsion-freeness of $\operatorname{Pic}(X)$ then forces $\varphi^*\mathcal{O}_X(1) \cong \mathcal{O}_X(1)$.

It remains to treat the case $n = 1$ and $r = 2$, so that $X$ is a smooth complete intersection curve in $\PP^3$ of multidegree $(d_1, d_2)$ with $3 \leq d_1 < d_2$. Let $d = d_1 d_2$ be the degree of the curve. Recall that a linear series of projective dimension $3$ and degree $d$ on $X$ is classically denoted $g^3_d$. Let $|H|$ be the complete linear series of hyperplane sections associated to $\mathcal{O}_X(1)$, which constitutes a $g^3_d$. For any $\varphi \in \Aut(X)$, the pullback $\varphi^*\mathcal{O}_X(1)$ defines a linear series $W$ with the same numerical invariants as $|H|$.

If $3 \leq d_1 \leq 4$, then \cite[Corollary~2.5]{CL84} guarantees that $|H|$ is the unique $g^3_d$ on $X$, so $W = |H|$. If $d_1 \geq 5$, then \cite[Theorem~2.6]{CL84} gives $W \leq |H|$, meaning there exists an effective divisor $E \geq 0$ on $X$ such that $W + E \subseteq |H|$. Taking degrees yields \[\deg W + \deg E \leq \deg |H|.\] Since $W$ and $|H|$ have the same degree $d$, it follows that $\deg E \leq 0$. As $E$ is effective, we conclude $E = 0$, so $W \subseteq |H|$ as a vector space inclusion. Since both linear series have projective dimension $3$, this forces $W = |H|$. In either case, $W = |H|$ implies $\varphi^*\mathcal{O}_X(1) \cong \mathcal{O}_X(1)$.

The isomorphism $\varphi^*\mathcal{O}_X(1) \cong \mathcal{O}_X(1)$ induces a linear automorphism of $H^0(X, \mathcal{O}_X(1))$ which, via \eqref{eq:rest}, lifts uniquely to a linear automorphism of $H^0(\PP^{n+r}, \mathcal{O}_{\PP^{n+r}}(1))$, defining an element of $\PGL(V)$ whose restriction to $X$ is $\varphi$. Hence $\Aut(X) = \AutL(X)$.
\end{proof}

The condition of \cref{prop:AutL} admits a formulation in terms of the genus for curves in $\PP^3$.

\begin{corollary}
Let $X = V(F_1) \cap V(F_2) \subset \PP^3$ be a smooth complete intersection curve of multidegree $(d_1, d_2)$ with $3 \leq d_1 < d_2$. Then its genus satisfies $g \geq 2$, and every automorphism of $X$ is linear, that is, $\Aut(X) = \AutL(X)$.
\end{corollary}

\begin{remark}
The extension of \cref{prop:AutL} to smooth complete intersection curves in higher-dimensional projective spaces ($n = 1$ and $r \geq 3$) gives obstructions that prevent the direct application of the methods used in its proof. For surfaces, the Picard group is torsion-free by the Lefschetz hyperplane theorem, which is the key property exploited in the $n = 2$ case. For a curve $X$, however, $\operatorname{Pic}(X) \cong \operatorname{Jac}(X) \oplus \ZZ$, and the Jacobian contains $k$-torsion for every $k \geq 2$, so the relation $k \cdot [ \varphi^*\mathcal{O}_X(1)] = k \cdot [\mathcal{O}_X(1)]$ no longer forces equality of the line bundles. The linear series argument used for spatial curves is likewise obstructed: the results of Ciliberto and Lazarsfeld \cite{CL84} apply specifically to curves of codimension two in $\PP^3$, and, to the best of the authors' knowledge, analogous results bounding the scheme of linear series $g^r_d(X)$ for complete intersection curves in higher codimension ($r \geq 3$) are not currently available in the literature.
\end{remark}

\medskip

We now recall some definitions and results of \cite{GL13,GALMVL24} that will be used throughout. The differential method introduced by Oguiso and Yu in \cite{OY19} and developed in full generality in \cite{GALMVL24} constrains the automorphisms of a smooth hypersurface $V(F)$ through two combinatorial invariants of $F$: the sparsity of its monomial support, and a partial order on the variables induced by its partial derivatives. To define them we first introduce some monomial notation. For an exponent vector $a = (a_0, \ldots, a_{n+r}) \in \ZZ_{\geq 0}^{n+r+1}$ we let $x^a := x_0^{a_0} \cdots x_{n+r}^{a_{n+r}}$ denote the corresponding monomial, and for a nonzero homogeneous polynomial $F \in S(V^*)$ we denote by
\[
M(F) \;:=\; \bigl\{\, a \in \ZZ_{\geq 0}^{n+r+1} \;:\; \text{the coefficient of } x^a \text{ in } F \text{ is nonzero}\,\bigr\}
\]
the \emph{monomial support} of $F$.

\begin{definition}[{\cite[Def.~2.2, 2.4]{GALMVL24}}]
\label{def:spar-order}
Let $F \in S^d(V^*)$ be a nonzero homogeneous polynomial.
\begin{enumerate}[\rm(i)]
\item The \emph{$\ell^1$-distance} between two monomials $x^a$ and $x^b$ is $\|a-b\|_1 = \sum_i |a_i - b_i|$.
\item The \emph{sparsity} of $F$ is $\spar(F) := \min\{\|a-b\|_1 : a, b \in M(F),\ a \neq b\}$.
\item The \emph{variables} of $F$ are $\vars(F) := \{x_i \in \beta^* : x_i \text{ appears in } F\}$.
\item The relation $\leq_F$ on $\beta^*$ is defined by $x_i \leq_F x_j$ if and only if $\vars(\partial F/\partial x_i) \subseteq \vars(\partial F/\partial x_j)$.
\end{enumerate}
\end{definition}

The sparsity $\spar(F)$ is always an even nonnegative integer. The relation $\leq_F$ is reflexive and transitive but need not be antisymmetric; we say that $(\beta^*, \leq_F)$ is a poset when it is, and we say that the poset is trivial when $x_i \leq_F x_j$ implies $i = j$.

We will denote by $\operatorname{GP}(V, \beta) \subseteq \GL(V)$ the subgroup of those automorphisms of $V$ whose matrix in the basis $\beta$ is a generalized permutation matrix, i.e., a matrix with at most one nonzero entry in each row and each column. Similarly, $\operatorname{GT}(V, \beta) \subseteq \GL(V)$ denotes the subset of those automorphisms whose matrix in $\beta$ is of the form $P_1 T P_2$, where $P_1, P_2$ are permutation matrices and $T$ is upper triangular; we call such matrices generalized triangular. We write $\PGP(V, \beta)$ and $\PGT(V, \beta)$ for the images of $\operatorname{GP}(V, \beta)$ and $\operatorname{GT}(V, \beta)$, respectively, in $\PGL(V)$. Note that every generalized permutation matrix is the product of a diagonal matrix and a permutation matrix, so that $\operatorname{GP}(V,\beta) \subseteq \operatorname{GT}(V,\beta)$ and consequently $\PGP(V,\beta) \subseteq \PGT(V,\beta)$.

With this notation, the differential method yields the following.

\begin{theorem}[{\cite[Thm.~2.6, Cor.~2.8]{GALMVL24}}]
\label{thm:galm22-shape}
Let $Y = V(F) \subset \PP^{n+r}$ be a smooth hypersurface of dimension $n+r-1 \geq 1$ and degree $d \geq 3$, with $(\dim Y, d) \neq (1, 3), (2, 4)$.  If           $\spar(F) > 4$ and $(\beta^*, \leq_F)$ is a poset, then $\Aut(Y) \subseteq \PGT(V, \beta)$.  If moreover the poset is trivial, then $\Aut(Y) \subseteq \PGP(V, \beta)$.
\end{theorem}

We will also need the following result on the order of an automorphism of a smooth hypersurface, in the form established in \cite[Thm.~2.1]{GALM22}, see also \cite[Rem.~2.2]{GALM22}.

\begin{proposition}[{\cite[Thm.~2.1]{GALM22}}]
\label{thm:GL-order}
Let $Y = V(F) \subset \PP^{n+r}$ be a smooth hypersurface of dimension $n+r-1 \geq 1$ and degree $d \geq 3$, with $(\dim Y, d) \neq (1, 3), (2, 4)$.  Suppose $Y$ admits an automorphism of order $q = p^s$, with $p$ a prime not dividing $d$.  Then there exists $\ell \in \{1, \dots, n+r+1\}$ such that $(1-d)^\ell \equiv 1 \pmod{q}$.
\end{proposition}

\section{Automorphisms of complete intersections}
\label{sec:aut-CI}

In this section we apply the consequences of the differential method from \cite{GALMVL24} and the order criterion of \cite{GL13,GALM22} to study complete intersections $X = V(F_1) \cap \cdots \cap V(F_r) \subset \PP^{n+r}$ of $r$ hypersurfaces of multidegree $(d_1, \ldots, d_r)$ with $d_1 < d_2 < \cdots < d_r$.  The starting point is that the lowest-degree generator $F_1$ is automatically preserved up to a scalar by every automorphism of $X$; under a combinatorial rigidity condition on the tuple $(F_1, \ldots, F_r)$, the same holds inductively for each $F_j$, and the study of $\Aut(X)$ then reduces to that of $\Aut(V(F_1)) \cap \cdots \cap \Aut(V(F_r))$.

We retain the notation of \cref{sec:prelim} and write $\IX = \langle F_1, \ldots, F_r \rangle$ for the homogeneous ideal of $X$.  In this section, we assume that $X$ is a complete intersection, so that $F_1, \ldots, F_r$ form a regular sequence in $S(V^*)$.

\begin{lemma}
\label{lem:F1-preserved}
Let $X = V(F_1) \cap \cdots \cap V(F_r) \subset \PP^{n+r}$ be a complete intersection of multidegree $(d_1, \ldots, d_r)$ with $d_1 < d_2 < \cdots < d_r$.  For every $\varphi \in \AutL(X)$ and every lift $\tvarphi \in \GL(V)$, there exists $a_1 \in \CC^*$ such that $\tvarphi^*(F_1) = a_1 F_1$.  In particular, $\varphi \in \Aut(V(F_1))$.
\end{lemma}

\begin{proof}
Since $d_1 < d_j$ for all $j \geq 2$, the graded component $(\IX)_{d_1} = \CC \cdot F_1$ is one-dimensional.  The lift $\tvarphi^*$ preserves $\IX$ and hence each of its graded components, so $\tvarphi^*(F_1) = a_1 F_1$ for some $a_1 \in \CC$.  Since $\tvarphi^*$ is invertible, $a_1 \neq 0$.
\end{proof}

Combined with \cref{prop:AutL}, \cref{lem:F1-preserved} reduces the study of $\Aut(X)$ to that of automorphisms of the hypersurface $V(F_1)$.

\begin{corollary}
\label{thm:CI-shape}
Let $X = V(F_1) \cap \cdots \cap V(F_r) \subset \PP^{n+r}$ be a complete intersection of dimension $n \geq 1$ and multidegree $(d_1, \ldots, d_r)$ with $3 \leq d_1 < d_2 < \cdots < d_r$ and $(\dim V(F_1), d_1) \neq (2, 4)$.  Assume either $n \geq 2$, or $n = 1$, $r = 2$ and $X$ smooth.  Assume further that $V(F_1)$ is smooth, $\spar(F_1) > 4$, and $(\beta^*, \leq_{F_1})$ is a poset.  Then $\Aut(X) \subseteq \PGT(V, \beta)$, and if the poset is trivial, $\Aut(X) \subseteq \PGP(V, \beta)$.
\end{corollary}

\begin{proof}
Every automorphism of $X$ is linear by \cref{prop:AutL}, so $\Aut(X)=\AutL(X)$, and \cref{lem:F1-preserved} gives $\Aut(X) \subseteq \Aut(V(F_1))$.  The conclusion now follows from \cref{thm:galm22-shape} applied to $V(F_1)$.
\end{proof}

We come to the main result of this section: under a combinatorial condition on the tuple $(F_1, \ldots, F_r)$, every automorphism of $X$ preserves each generator up to a scalar, and the automorphism group of $X$ is recovered as the intersection of the automorphism groups of the $r$ hypersurfaces.

\begin{definition}
\label{def:rigidity}
Let $F \in S(V^*)$ be a nonzero homogeneous polynomial.  The \emph{generalized-permutation orbit} of its monomial support is
\[
\mathcal{O}(F) \;:=\; \bigcup_{\sigma \in \mathfrak{S}_{n+r+1}} \bigl\{\, \sigma \cdot a : a \in M(F)\,\bigr\} \;\subseteq\; \ZZ_{\geq 0}^{n+r+1},
\]
where $\sigma \cdot a$ denotes the exponent vector obtained by permuting the entries of $a$ according to $\sigma$, and we write $S_{\mathcal{O}(F)} \subseteq S(V^*)$ for the linear span of the monomials with exponent in $\mathcal{O}(F)$.  We say that the tuple $(F_1, \ldots, F_r)$ with $d_1 < d_2 < \cdots < d_r$ is \emph{rigid} (with respect to the basis $\beta$) if, for every $j \in \{2, \ldots, r\}$,
\[
\Bigl(\sum_{i < j} F_i \cdot S^{d_j - d_i}(V^*)\Bigr) \cap S_{\mathcal{O}(F_j)} \;=\; \{0\}.
\]
\end{definition}

\begin{remark}
\label{rmk:rigidity-r2}
In the case $r = 2$, which is the setting needed for the Fermat and Klein complete intersections studied in \cref{sec:fermat,sec:klein}, the condition reduces to a single requirement: the pair $(F_1, F_2)$ with $d_1 < d_2$ is rigid if and only if, for every $Q \in S^{d_2 - d_1}(V^*) \setminus \{0\}$, the product $F_1 \cdot Q$ has at least one monomial whose exponent lies outside $\mathcal{O}(F_2)$.
\end{remark}

Like sparsity, rigidity is a combinatorial condition on monomial supports that depends on the choice of basis $\beta$. Both conditions are useful precisely for sparse polynomials such as Fermat and Klein polynomials, where the monomial support is small. They play parallel roles in our setting: sparsity of $F_1$ and triviality of the poset $(\beta^*,\leq_{F_1})$ force the shape of automorphisms of $V(F_1)$, while rigidity of $(F_1, \ldots, F_r)$ then forces each $F_j$ to be preserved up to a scalar.

\begin{theorem}
\label{thm:CI-main}
Let $X = V(F_1) \cap \cdots \cap V(F_r) \subset \PP^{n+r}$ be a complete intersection of dimension $n \geq 1$ and multidegree $(d_1, \ldots, d_r)$ with $3 \leq d_1 < d_2 < \cdots < d_r$ and $(\dim V(F_1), d_1) \neq (2, 4)$.  Assume either $n \geq 2$, or $n = 1$, $r = 2$ and $X$ smooth.  Assume further that $V(F_1)$ is smooth, $\spar(F_1) > 4$, the poset $(\beta^*, \leq_{F_1})$ is trivial, and the tuple $(F_1, \ldots, F_r)$ is rigid. Then
\[
\Aut(X) = \bigcap_{j=1}^{r} \Aut(V(F_j)),
\]
as subgroups of $\PGL(V)$.  In particular, every $\varphi \in \Aut(X)$ admits a lift $\tvarphi \in \GL(V)$ such that $\tvarphi^*(F_j) = a_j F_j$ for some $a_j \in \CC^*$ and every $j \in \{1, \ldots, r\}$.
\end{theorem}

\begin{proof}
The inclusion $\bigcap_{j=1}^{r} \Aut(V(F_j)) \subseteq \Aut(X)$ is immediate, since any linear automorphism preserving each $V(F_j)$ also preserves their intersection.

For the reverse inclusion, let $\varphi \in \Aut(X)$.  By \cref{thm:CI-shape}, $\tvarphi$ is a generalized permutation matrix in $\beta$: there exist $\sigma \in \mathfrak{S}_{n+r+1}$ and $\mu_i \in \CC^*$ with $\tvarphi^*(x_i) = \mu_i\, x_{\sigma(i)}$.  For any monomial $x^a \in S(V^*)$,
\[
\tvarphi^*(x^a) = \lambda_a\, x^{\sigma \cdot a}, \quad \lambda_a = \prod_i \mu_i^{a_i} \in \CC^*,
\]
so $M(\tvarphi^*(F_j)) = \sigma \cdot M(F_j) \subseteq \mathcal{O}(F_j)$ for every $j$.

We now prove by induction on $j \in \{1, \ldots, r\}$ that $\tvarphi^*(F_j) = a_j F_j$ for some $a_j \in \CC^*$.  The case $j = 1$ is \cref{lem:F1-preserved}.  Assume the claim holds for all indices less than $j$, with $j \geq 2$.  Since $\varphi$ preserves $(\IX)_{d_j}$ and the $F_i$ form a regular sequence with $d_i < d_j$ for $i < j$, there exist $a_j \in \CC$ and $Q_i \in S^{d_j - d_i}(V^*)$ with
\[
\tvarphi^*(F_j) = a_j F_j + \sum_{i < j} F_i \cdot Q_i.
\]
Both $a_j F_j$ and $\tvarphi^*(F_j)$ have monomial support in $\mathcal{O}(F_j)$, hence so does $\sum_{i < j} F_i Q_i$.  Rigidity of $(F_1, \ldots, F_r)$ at level $j$ then forces $\sum_{i < j} F_i Q_i = 0$, so $\tvarphi^*(F_j) = a_j F_j$.  Since $\tvarphi$ is invertible, $a_j \in \CC^*$.  This completes the induction and shows $\varphi \in \bigcap_{j=1}^{r} \Aut(V(F_j))$.
\end{proof}

\begin{example}
\label{ex:fermat-klein-sextic}
To illustrate \cref{thm:CI-main}, consider in $\PP^3$ the complete intersection $X = V(F_1) \cap V(F_2)$ defined by the Fermat cubic and the Klein sextic surfaces:
\[
F_1 \;=\; x_0^3 + x_1^3 + x_2^3 + x_3^3,
\qquad
F_2 \;=\; x_0^5 x_1 + x_1^5 x_2 + x_2^5 x_3 + x_3^5 x_0.
\]
The hypersurfaces $V(F_1)$, $V(F_2)$ and their intersection $X$ are all smooth.  A straightforward computation yields $\spar(F_1) = 6 > 4$, $\spar(F_2) = 10 > 4$, and both $(\beta^*, \leq_{F_1})$ and $(\beta^*, \leq_{F_2})$ are trivial posets.

To verify that $(F_1, F_2)$ is rigid, note that $\mathcal{O}(F_2) = \{5 e_i + e_j : i \neq j\}$ consists of exponents with support of size $2$ and nonzero entries $\{5, 1\}$.  Given $Q \in S^3(V^*) \setminus \{0\}$, choose a monomial $x_i x_j x_k$ of $Q$ with nonzero coefficient $q$ (the indices $i,j,k$ could potentially be equal), and let $l \in \{0, 1, 2, 3\}$ be any index with $l \notin \{i, j, k\}$ (possible since $n+3 = 4$ and the support of $\{i, j, k\}$ has size at most $3$).  The monomial $x_l^3 \cdot x_i x_j x_k$ appears in $F_1 \cdot Q$ with coefficient $q$, and its exponent $3 e_l + e_i + e_j + e_k$ has support of size at least $2$ if $i=j=k$, however has a nonzero entry equal to $3$, hence lies outside $\mathcal{O}(F_2)$. By \cref{thm:CI-main},
\[
\Aut(X) \;=\; \Aut(V(F_1)) \cap \Aut(V(F_2)).
\]
By \cref{prop:aut-fermat}, $\Aut(V(F_1)) \cong (\ZZ/3\ZZ)^3 \rtimes \mathfrak{S}_4$, and by \cref{prop:aut-klein}, $\Aut(V(F_2)) \cong (\ZZ/104\ZZ) \rtimes (\ZZ/4\ZZ)$. Since $\gcd(3, 104) = 1$, the diagonal part of the intersection is trivial, and the only common element is the cyclic permutation $\nu : (x_0 : x_1 : x_2 : x_3) \mapsto (x_1 : x_2 : x_3 : x_0)$, so 
\[
\Aut(X) \;\cong\; \ZZ/4\ZZ.
\]
The curve $X$ has degree $d_1 d_2 = 18$ and genus  
\[
g(X) = 1 + \frac{d_1 d_2 (d_1 + d_2 - 4)}{2}=46,
\]
by \cite[Exer.~II.8.4(g)]{Har77}. Since the genus of a smooth plane curve of degree $d$ is $(d-1)(d-2)/2$ \cite[Exer.~I.7.2(b)]{Har77}, and the equation $(d-1)(d-2)/2 = 46$ has no integer solution, the curve $X$ is not isomorphic to any smooth hypersurface in $\PP^2$.
\end{example}

We close the section with an order constraint for automorphisms of $X$.  Under the hypotheses of \cref{thm:CI-main}, each $F_j$ is preserved up to a scalar by every lift of an automorphism of $X$, so \cref{thm:GL-order} applies to each of the $r$ hypersurfaces.

\begin{proposition}
\label{prop:order}
Let $X = V(F_1) \cap \cdots \cap V(F_r) \subset \PP^{n+r}$ be a complete intersection of dimension $n \geq 1$ and multidegree $(d_1, \ldots, d_r)$ with $3 \leq d_1 < d_2 < \cdots < d_r$. Assume either $n \geq 2$, or $n = 1$, $r = 2$, $X$ smooth, and $d_j \neq 4$ for $j \in \{1, 2\}$. Assume further that each $V(F_j)$ is smooth, $\spar(F_1) > 4$, the poset $(\beta^*, \leq_{F_1})$ is trivial, and the tuple $(F_1, \ldots, F_r)$ is rigid. Suppose $X$ admits an automorphism $\varphi$ of order $q = p^s$, with $p$ a prime such that $p \nmid d_1 \cdots d_r$. Then for every $j \in \{1, \ldots, r\}$ there exists $\ell_j \in \{1, \dots, n+r+1\}$ such that
\[
(1 - d_j)^{\ell_j} \equiv 1 \pmod{q}.
\]
\end{proposition}

\begin{proof}
By \cref{thm:CI-main}, $\varphi \in \bigcap_{j=1}^{r} \Aut(V(F_j))$.  The assumption $d_j \neq 4$ when $n = 1$ ensures that $(\dim V(F_j), d_j) \neq (2,4)$, so that \cref{thm:GL-order} applies to each $V(F_j)$ and yields $\ell_j$.
\end{proof}

\section{Fermat complete intersections}
\label{sec:fermat}

In this section we apply \cref{sec:aut-CI} to complete intersections of two Fermat hypersurfaces.  We first recall the Fermat hypersurface and its automorphism group, then determine when the intersection $X = V(F_1) \cap V(F_2)$ of two Fermat hypersurfaces of distinct degrees is smooth.

\begin{definition}
\label{def:fermat}
The \emph{Fermat hypersurface} of degree $d$ in $\PP^{n+2}$ is the hypersurface $V(F) \subset \PP^{n+2}$ defined by
\[
F \;=\; x_0^d + x_1^d + \cdots + x_{n+2}^d.
\]
\end{definition}

The Fermat hypersurface is smooth by a direct application of the Jacobian criterion, and its automorphism group is known.

\begin{proposition}[{\cite[Prop.~3.1]{GALMVL24}}]
\label{prop:aut-fermat}
Let $V(F) \subset \PP^{n+2}$ be the Fermat hypersurface of degree $d \geq 3$, with $(n, d) \neq (0, 3), (1, 4)$.  Then
\[
\Aut(V(F)) \;\cong\; (\ZZ/d\ZZ)^{n+2} \rtimes \mathfrak{S}_{n+3},
\]
where $\mathfrak{S}_{n+3}$ acts by permutation of the variables and $(\ZZ/d\ZZ)^{n+2}$ acts by  multiplications of $n+2$ of the variables by $d$-th roots of unity up to projective equivalence.
\end{proposition}

We now turn to the complete intersection
\[
X \;:=\; V(F_1) \cap V(F_2) \;\subset\; \PP^{n+2},
\qquad
F_1 \;=\; x_0^{d_1} + \cdots + x_{n+2}^{d_1},
\qquad
F_2 \;=\; x_0^{d_2} + \cdots + x_{n+2}^{d_2},
\]
of two Fermat hypersurfaces of different degrees $d_1, d_2 \geq 2$. Throughout the section we assume $d_1 < d_2$, and we set
\[
m \;:=\; d_2 - d_1, \qquad m' \;:=\; \frac{m}{\gcd(m, d_1)}.
\]
Note that $\gcd(m, d_1) = \gcd(d_2 - d_1, d_1) = \gcd(d_1, d_2)$, so $m' = m / \gcd(d_1, d_2)$. We now give conditions for the smoothness of $X$.

The smoothness of $X$ rests on the following combinatorial lemma about vanishing sums of roots. For $t \geq 1$, write $\mu_t \subset \CC^*$ for the group of $t$-th roots of unity, and define
\[
\sigma(t) \;:=\; \min\Bigl\{ s \geq 1 \;\Big|\; \text{there exist } \beta_1, \ldots, \beta_s \in \mu_t \text{ with } \sum_{i=1}^{s} \beta_i = 0 \Bigr\},
\]
with the convention $\sigma(1) = +\infty$.
\begin{lemma}[Lam--Leung]
\label{lem:sigma-min-prime}
For every integer $t \geq 2$, $\sigma(t) = p$, where $p$ is the smallest prime divisor of $t$.
\end{lemma}

\begin{proof}
This is a special case of the main theorem of \cite{LaLe00}, which characterizes the set of weights of vanishing sums of $t$-th roots of unity as the numerical semigroup $\ZZ_{\geq 0}\, p_1 + \cdots + \ZZ_{\geq 0}\, p_s$, where $p_1, \ldots, p_s$ are the distinct prime divisors of $t$; the smallest nonzero element of this semigroup is the smallest prime divisor of $t$.
\end{proof}

By \cite[Proposition~3.2]{Kon02}, the complete intersection $X$ is geometric, i.e., irreducible and reduced. In the following propositions, we give necessary and sufficient conditions for $X$ to be a smooth complete intersection. 

\begin{proposition}
\label{prop:fermat-CI-smooth}
Let $X = V(F_1) \cap V(F_2) \subset \PP^{n+2}$ be the complete intersection of two Fermat hypersurfaces of degrees $d_1 < d_2$ with $d_1 \geq 2$, and let $m, m'$ be as above. Then, $X$ is smooth  if and only if either $m' = 1$, or $m' \geq 2$ and $n + 3 < p$, where $p$ is the smallest prime divisor of $m'$.
\end{proposition}

\begin{proof}
Let $P=[x_0:\cdots:x_{n+2}] \in \PP^{n+2}$.
Since both Fermat hypersurfaces $V(F_1)$ and $V(F_2)$ are smooth, the Jacobian criterion shows that $P$ is a singular point of $X$ if and only if $P\in X$ and there exists $\lambda\in\CC^{*}$ such that
\[
d_1x_i^{d_1-1}
=
\lambda d_2x_i^{d_2-1}
\qquad
(i=0,\dots,n+2).
\]
Equivalently,
\[
x_i=0
\quad\text{or}\quad
\lambda=\frac{d_1}{d_2}x_i^{-m},
\]
where $m=d_2-d_1$. Hence, for any two indices $i,j$ with $x_i,x_j\neq 0$, we have $x_i^m=x_j^m$. Suppose that $P\in \PP^{n+2}$ satisfies this last condition and let
\[
I=\{i : x_i\neq 0\}.
\]
Fix $\xi \in \CC^*$ with $\xi^m$ equal to the common value $x_i^m$ for $i \in I$, and write $x_i = \alpha_i \xi$ where $\alpha_i \in \mu_m$.  Since $d_2 = d_1 + m$ and $\alpha_i^m = 1$, we have $\alpha_i^{d_2} = \alpha_i^{d_1}$, so for $j = 1, 2$,
\[
F_j(P) \;=\; \xi^{d_j} \sum_{i \in I} \alpha_i^{d_1}.
\]
Therefore, a point $P\in\PP^{n+2}$ satisfying the Jacobian condition above is singular on $X$ if and only if
\[
\sum_{i\in I}\alpha_i^{d_1}=0,
\qquad
\alpha_i\in\mu_m.
\]
Consider the homomorphism $\mu_m \to \mu_m$, $\alpha \mapsto \alpha^{d_1}$.  Its image is $\mu_{m'}$. Indeed, $(\alpha^{d_1})^{m'} = \alpha^{d_1 m / \gcd(m, d_1)} = (\alpha^m)^{d_1 / \gcd(m, d_1)} = 1$, and conversely $\gcd(d_1 / \gcd(m, d_1),\, m') = 1$ shows the induced map $\mu_m \to \mu_{m'}$ is surjective.  We conclude that
\[
X \text{ is singular} \;\iff\; \text{there exist } 2 \leq s \leq n+3 \text{ and } \beta_1, \ldots, \beta_s \in \mu_{m'} \text{ with } \sum_{i=1}^{s} \beta_i = 0,
\]
which by definition of $\sigma(m')$ is equivalent to $\sigma(m') \leq n + 3$.  Hence $X$ is smooth if and only if $n + 3 < \sigma(m')$. If $m' = 1$, then $\mu_{m'} = \{1\}$ and the condition $\sum_{i=1}^{s}\beta_i = 0$ becomes a sum of $s$ ones, which never vanishes; hence $\sigma(m') = +\infty$ and $X$ is smooth.  If $m' \geq 2$, \cref{lem:sigma-min-prime} gives $\sigma(m') = p$, the smallest prime divisor of $m'$, and the conclusion follows.
\end{proof}

The combinatorial input required to apply \cref{thm:CI-main} is the rigidity of the pair $(F_1, F_2)$, which we prove in the following lemma.

\begin{lemma}
\label{lem:fermat-rigidity}
Let $F_1, F_2 \in S(V^*)$ be Fermat polynomials of degrees $3 \leq d_1 < d_2$ in $n + 3$ variables, with $n \geq 1$. Then the pair $(F_1, F_2)$ is rigid in the standard basis $\beta$.
\end{lemma}

\begin{proof}
The orbit $\mathcal O(F_2)$ consists precisely of the pure-power exponents
$\{d_2 e_i : 0\le i \le n+2\}$. Suppose, for contradiction, that all
monomials of $F_1 Q$ lie in this orbit for some nonzero
$Q\in S^{d_2-d_1}(V^*)$. Then
\begin{align} \label{eq:lemma3.5}
F_1 Q = \sum_{i=0}^{n+2} a_i x_i^{d_2}, \qquad a_i\in\CC.
\end{align}
Write $d_2 = q d_1 + r$ with $0\le r<d_1$. View $\CC[x_0,\ldots,x_{n+2}]$
as $\CC[x_0,\ldots,x_{n+1}][x_{n+2}]$. Since
\[
F_1 = x_{n+2}^{d_1} + h, \qquad h := \sum_{i=0}^{n+1} x_i^{d_1}
\in \CC[x_0,\ldots,x_{n+1}],
\]
is monic in $x_{n+2}$ and $F_1$ divides
$\sum_{i=0}^{n+2} a_i x_i^{d_2}$, the remainder of this latter polynomial modulo $F_1$ must vanish. Using $x_{n+2}^{d_1} \equiv -h \pmod{F_1}$, we have
$x_{n+2}^{d_2} \equiv (-h)^q x_{n+2}^r \pmod{F_1}$. Since the right hand side  has degree $r < d_1$ with respect to $x_{n+2}$, replacing this congruence in \eqref{eq:lemma3.5} yields the following unconditional identity in the polynomial ring
\[
\sum_{i=0}^{n+1} a_i x_i^{d_2} + a_{n+2} (-h)^q x_{n+2}^r = 0
\quad \text{in}\quad  \CC[x_0,\ldots,x_{n+1}][x_{n+2}].
\]
If $r>0$, equating coefficients of $x_{n+2}^0$ and $x_{n+2}^r$ gives
$a_i=0$ for all $i\le n+1$ and $a_{n+2}=0$. If $r=0$, then $d_2 = q d_1$
with $q\ge 2$ (since $d_2>d_1$), and the identity becomes
\[
\sum_{i=0}^{n+1} a_i x_i^{q d_1} = -a_{n+2}(-h)^q
\quad \text{in } \CC[x_0,\ldots,x_{n+1}].
\]
When $a_{n+2}\ne 0$, the right-hand side contains the mixed monomial
$x_0^{d_1(q-1)} x_1^{d_1}$ with nonzero coefficient (since $n\ge 1$
ensures $h$ has at least two terms), which cannot appear on the left.
Hence $a_{n+2}=0$, and then $a_i=0$ for all $i$. Thus $F_1Q=0$, and since
the polynomial ring is a domain, $Q=0$, contradicting the choice of $Q$.
\end{proof}

We can now apply \cref{thm:CI-main} to determine the automorphism group of any Fermat complete intersection.

\begin{theorem}
\label{thm:fermat-CI-aut}
Let $X = V(F_1) \cap V(F_2) \subset \PP^{n+2}$ be a complete intersection of two Fermat hypersurfaces of degrees $3 \leq d_1 < d_2$, with $n \geq 1$ and $(n, d_j) \neq (1, 4)$ for $j = 1, 2$.  Assume either $n \geq 2$, or $n = 1$ and $X$ smooth.  Then
\[
\Aut(X) \;\cong\; (\ZZ/\gcd(d_1, d_2)\ZZ)^{n+2} \rtimes \mathfrak{S}_{n+3},
\]
acting in the standard basis by generalized permutation matrices.
\end{theorem}

\begin{proof}
A straightforward computation yields $\spar(F_1) = 2 d_1 \geq 6 > 4$ and $(\beta^*, \leq_{F_1})$ is the trivial poset.  The pair $(F_1, F_2)$ is rigid by \cref{lem:fermat-rigidity}.  \cref{thm:CI-main} then gives
\[
\Aut(X) \;=\; \Aut(V(F_1)) \cap \Aut(V(F_2))
\]
as subgroups of $\PGL(V)$.  By \cref{prop:aut-fermat}, both factors are subgroups of $\PGP(V, \beta)$ acting by generalized permutation matrices: the symmetric group $\mathfrak{S}_{n+3}$ permutes variables in both factors, while a diagonal lift $\tvarphi = \diag(\lambda_0, \ldots, \lambda_{n+2})$ preserves $F_j$ if and only if $\lambda_i^{d_j}$ is independent of $i$.  After normalizing $\lambda_0 = 1$, this becomes $\lambda_i^{d_j} = 1$ for all $i$ and $j = 1, 2$, equivalently $\lambda_i^{\gcd(d_1, d_2)} = 1$.  Hence the diagonal part of $\Aut(X)$ is $(\ZZ/\gcd(d_1, d_2)\ZZ)^{n+2}$.
\end{proof}

\section{Klein complete intersections}
\label{sec:klein}

In this section we apply \cref{sec:aut-CI} to complete intersections of two Klein hypersurfaces obtained via the reverse-order construction.  We first recall the Klein hypersurface and its automorphism group, then introduce the reverse-order construction.

\begin{definition}
\label{def:klein}
The \emph{Klein hypersurface} of degree $d$ in $\PP^{n+2}$ is the hypersurface $V(K) \subset \PP^{n+2}$ defined by
\[
K \;=\; x_0^{d-1} x_1 + x_1^{d-1} x_2 + \cdots + x_{n+1}^{d-1} x_{n+2} + x_{n+2}^{d-1} x_0.
\]
\end{definition}

The Klein hypersurface is smooth if $d\geq 3$ by \cite[Lem.~3.1]{GL13}, and its automorphism group is known in nearly all cases.

\begin{proposition}[{\cite[Prop.~3.3, Thm.~3.5]{GALMVL24}}, {\cite{Adl78}}, {\cite{Dol12}}]
\label{prop:aut-klein}
Let $V(K) \subset \PP^{n+2}$ be the Klein hypersurface of degree $d \geq 3$
and dimension $n+1$. Assume $
(n+1,d)\notin \{(1,3),\,(2,3),\,(2,4),\,(3,3)\}
$.
Then
\[
\Aut(V(K)) \;\cong\; (\ZZ/m\ZZ) \rtimes (\ZZ/(n+3)\ZZ),
\qquad
m = \frac{(d-1)^{n+3} - (-1)^{n+3}}{d},
\]
where the cyclic factor of order $m$ is generated by the diagonal automorphism
\[
\varphi_0 \;=\; \diag\bigl(\zeta_m,\, \zeta_m^{1-d},\,
\zeta_m^{(1-d)^2},\, \ldots,\, \zeta_m^{(1-d)^{n+2}}\bigr),
\]
and the cyclic factor of order $n+3$ is generated by the cyclic permutation
of the coordinates
\[
\nu \colon (x_0 : x_1 : \cdots : x_{n+2})
\;\longmapsto\;
(x_1 : x_2 : \cdots : x_{n+2} : x_0).
\]
In the exceptional case $(n,d)=(1,3)$, equivalently when $V(K)$ is the Klein
cubic surface, one has $\Aut(V(K))\cong \mathfrak S_5$. In the exceptional case $(n,d)=(2,3)$, equivalently when $V(K)$ is the Klein
cubic threefold, one has $\Aut(V(K))\cong \operatorname{PSL}_2(\mathbf F_{11})$.
\end{proposition}

By \cite[Thm.~3.7]{GL13}, every smooth hypersurface of dimension $n+1$ and degree $d$ admitting an automorphism of prime order $p > (d-1)^{n+1}$ is isomorphic to a Klein hypersurface, $n+3$ is prime, and 
$$p = m = \frac{(d-1)^{n+3} + 1}{d}.$$  
Such primes are called \emph{generalized Wagstaff primes} of base $d-1$, and the Klein hypersurfaces realizing them are said to be of \emph{Wagstaff type}.

The Klein polynomial $K$ admits a natural variant, obtained by reversing the cyclic order of the variables, that defines another smooth Klein hypersurface in $\PP^{n+2}$.  We will study the relationship between these two hypersurfaces and the complete intersections they cut out.

\begin{definition}
\label{def:reverse}
The \emph{reverse-order} Klein hypersurface of degree $d$ in $\PP^{n+2}$ is the hypersurface $V(K') \subset \PP^{n+2}$ defined by
\[
K' \;=\; x_1^{d-1} x_0 + x_2^{d-1} x_1 + \cdots + x_{n+2}^{d-1} x_{n+1} + x_0^{d-1} x_{n+2}.
\]
\end{definition}

Note that $V(K')$ is isomorphic to a Klein hypersurface, so it is smooth and its automorphism group is described by \cref{prop:aut-klein}.  As a subvariety of $\PP^{n+2}$ in the standard basis, $V(K')$ is in general different from $V(K)$.  We show in \cref{prop:reverse} that, under an arithmetic congruence relating their degrees, the standard semidirect product $G = (\ZZ/m\ZZ) \rtimes (\ZZ/(n+3)\ZZ)$ preserves both hypersurfaces.

\begin{proposition}
\label{prop:reverse}
Let $V(K) \subset \PP^{n+2}$ be the Klein hypersurface of dimension $n+1$ and degree $d \geq 3$, and let $V(K') \subset \PP^{n+2}$ be the reverse-order Klein hypersurface of degree $d' \geq 2$. Set
\[
m = \frac{(d-1)^{n+3} - (-1)^{n+3}}{d},
\qquad
G = (\ZZ/m\ZZ) \rtimes (\ZZ/(n+3)\ZZ),
\]
where $G \subseteq \PGL(V)$ is generated by the diagonal automorphism $\varphi_0 = \diag(\zeta_m^{\sigma_0}, \ldots, \zeta_m^{\sigma_{n+2}})$ with $\sigma_j = (1-d)^j$, and the cyclic permutation $\nu \colon (x_0 : \cdots : x_{n+2}) \mapsto (x_1 : \cdots : x_{n+2} : x_0)$. If $(d-1)(d'-1) \equiv 1 \pmod{m}$, then $d' > d$ and
\[
G \subseteq \Aut(V(K)) \cap \Aut(V(K')).
\]
\end{proposition}

\begin{proof}
We first show that $d' > d$. Since $n \geq 1$, we have
\[
m \;\geq\; \frac{(d-1)^4 - 1}{d} \;>\; (d-1)^2,
\]
where the last inequality is equivalent to $(d-1)^4 - d(d-1)^2 - 1 > 0$, which holds for every $d \geq 3$. As $d' \geq 2$, the integer $(d-1)(d'-1)$ is at least $1$. If $d' \leq d$, then $(d-1)(d'-1) \leq (d-1)^2 < m$, so the congruence $(d-1)(d'-1) \equiv 1 \pmod{m}$ forces $(d-1)(d'-1) = 1$, hence $d = 2$, contradicting $d \geq 3$. Therefore $d' > d$.

For $j = 0, \ldots, n+1$,
\[
(d-1)\sigma_j + \sigma_{j+1} = (d-1)(1-d)^j + (1-d)^{j+1} = 0,
\]
and
\[
(d-1)\sigma_{n+2} + \sigma_0 = (d-1)(1-d)^{n+2} + 1 = 1 - (-1)^{n+3}(d-1)^{n+3},
\]
which is congruent to $0$ modulo $m$ by the definition of $m$; hence $\varphi_0^*(K) = K$. Multiplying each congruence $(d-1)\sigma_j + \sigma_{j+1} \equiv 0 \pmod{m}$ by $d'-1$ and using $(d-1)(d'-1) \equiv 1 \pmod{m}$ yields
\[
\sigma_j + (d'-1)\sigma_{j+1} \equiv 0 \pmod{m},
\]
which is exactly the condition that $\varphi_0$ fixes the reverse-order monomial $x_{j+1}^{d'-1} x_j$. Hence $\varphi_0$ fixes every monomial of $K'$, so $\varphi_0^*(K') = K'$. The cyclic permutation $\nu$ preserves both $V(K)$ and $V(K')$, by a straightforward verification. Since $\varphi_0$ and $\nu$ generate $G$ and both preserve $V(K)$ and $V(K')$, we conclude $G \subseteq \Aut(V(K)) \cap \Aut(V(K'))$.
\end{proof}

We now turn to the complete intersection 
\[
X \;:=\; V(K_1) \cap V(K_2) \;\subset\; \PP^{n+2},
\]
where $V(K_1)$ is a Klein hypersurface of degree $d_1$ and $V(K_2)$ is the reverse-order Klein hypersurface of degree $d_2$, so that
\begin{align*}
K_1 &\;=\; x_0^{d_1-1} x_1 + x_1^{d_1-1} x_2 + \cdots + x_{n+1}^{d_1-1} x_{n+2} + x_{n+2}^{d_1-1} x_0, \\
K_2 &\;=\; x_1^{d_2-1} x_0 + x_2^{d_2-1} x_1 + \cdots + x_{n+2}^{d_2-1} x_{n+1} + x_0^{d_2-1} x_{n+2}.
\end{align*}
Throughout the section we assume $d_1 < d_2$.  The intersection $X$ has dimension $n$, and we refer to $X$ as a \emph{Klein  intersection of multidegree $(d_1, d_2)$}. 

A first observation is that $X$ is always singular in odd dimension.

\begin{example}\label{ex:odd-singular}
Assume $n$ is odd, equivalently $n + 3$ is even, and consider the point
\[
P \;=\; [1 : 0 : 1 : 0 : \cdots : 1 : 0] \in \PP^{n+2}.
\]
A direct computation shows $K_1(P) = K_2(P) = 0$, so $P \in X$, and that $\nabla K_1(P) = [0 : 1 : 0 : 1 : \cdots : 0 : 1] = \nabla K_2(P)$.  Hence the gradients of $K_1$ and $K_2$ are proportional at $P$, so $X$ is singular at $P$.
\end{example}

In even dimension, $X$ may still have singular points, depending on the arithmetic of $d_1$ and $d_2$.

\begin{proposition}
\label{prop:singular-criterion}
Let $k$ be a divisor of $n + 3$ with $d_j \not\equiv 0 \pmod{k}$ for $j = 1, 2$, and let $\zeta_k$ be a primitive $k$-th root of unity.  Set
\[
P \;:=\; [1 : \zeta_k : \zeta_k^2 : \cdots : \zeta_k^{n+2}] \in \PP^{n+2}.
\]
Then $P \in X$, and $X$ is singular at $P$ if and only if
\[
d_1 \equiv d_2 \pmod{k}.
\]
\end{proposition}

\begin{proof}
Throughout the proof, all index arithmetic is cyclic modulo $n + 3$.  Using $d_j \not\equiv 0 \pmod{k}$, the sum of $k$-th roots of unity gives
\[
K_1(P) \;=\; \frac{n+3}{k}\, \zeta_k \sum_{j=0}^{k-1} \zeta_k^{j d_1} \;=\; 0,
\qquad
K_2(P) \;=\; \frac{n+3}{k}\, \zeta_k^{-1} \sum_{j=1}^{k} \zeta_k^{j d_2} \;=\; 0,
\]
so $P \in X$.  A point of $X$ is singular if and only if there exists $\lambda \in \CC$ with $\nabla K_1(P) = \lambda\, \nabla K_2(P)$, i.e., for every $j$,
\[
\lambda \;=\; \frac{(d_1 - 1) x_j^{d_1 - 2} x_{j+1} + x_{j-1}^{d_1 - 1}}{(d_2 - 1) x_j^{d_2 - 2} x_{j-1} + x_{j+1}^{d_2 - 1}}.
\]
Substituting $P$ yields
\[
\lambda \;=\; \zeta_k^{j(d_1 - d_2)} \cdot \frac{(d_1 - 1) \zeta_k + \zeta_k^{1 - d_1}}{(d_2 - 1) \zeta_k^{-1} + \zeta_k^{d_2 - 1}}.
\]
This is independent of $j$ if and only if $d_1 \equiv d_2 \pmod{k}$.
\end{proof}

Combining the two preceding statements yields necessary conditions for $X$ to be smooth.

\begin{corollary}
\label{cor:smooth-necessary}
If $X = V(K_1) \cap V(K_2) \subset \PP^{n+2}$ is smooth, then:
\begin{enumerate}[\rm(i)]
\item $n$ is even (equivalently, $n + 3$ is odd); and
\item for every divisor $k > 1$ of $n + 3$ with $d_j \not\equiv 0 \pmod{k}$ for $j = 1, 2$, one has $d_1 \not\equiv d_2 \pmod{k}$.
\end{enumerate}
\end{corollary}

\begin{proof}
Item (i) is \cref{ex:odd-singular}; item (ii) follows from \cref{prop:singular-criterion}.
\end{proof}

\begin{remark}
\label{rmk:smoothness-computational}
The arithmetic conditions in \cref{cor:smooth-necessary} fall short of sufficiency, and a purely algebraic proof of smoothness appears to be quite delicate.  When $n + 3 = p^k$ is a prime power, additional singular configurations arise from congruences modulo $p^k$ involving both $d_1$ and $d_2$, which we have detected computationally but not been able to characterize in closed form.  We have verified smoothness computationally for a range of small cases, summarized in \cref{tab:smoothness-data} below.
\end{remark}

\begin{table}[ht]
\renewcommand{\arraystretch}{1.2}
\begin{tabular*}{\linewidth}{@{\extracolsep{\fill}}clc@{}}
\toprule
$n + 3$ & range of $(d_1, d_2)$ tested & smoothness \\
\midrule
$5$ & $3 \leq d_1 \leq d_2 \leq 10$ & smooth except $(3,3), (4,4), (3,8), (4,9), (6,6), (7,7), (8,8), (9,9)$ \\
$7$ & $(3,3),\, (3,4),\, (3,5),\, (4,4)$ & smooth except $(3,3),\, (4,4)$ \\
\bottomrule
\end{tabular*}
\caption{Computational verification of smoothness for Klein complete intersections.}
\label{tab:smoothness-data}
\end{table}

We now turn to the automorphism group of a Klein complete intersection.  The combinatorial input required to apply \cref{thm:CI-main} is the rigidity of the pair $(K_1, K_2)$, which in this setting reduces to a direct comparison of monomial supports.

\begin{lemma}
\label{lem:klein-rigidity}
Let $K_1$ be the Klein polynomial of degree $d_1$ and let $K_2$ be the reverse-order Klein polynomial of degree $d_2$ in $n+3$ variables, with $n\geq 1$ and $3\leq d_1<d_2$. Then the pair $(K_1,K_2)$ is rigid in the standard basis.
\end{lemma}

\begin{proof}
Throughout the proof to simplify the notation, we consider all the indices $i$ in  $x_i$ to be reduced modulo $n+3$. The orbit
$\mathcal O(K_2)$ consists precisely of those exponent vectors of degree $d_2$
with support of size $2$ and nonzero entries $d_2-1$ and $1$.  For
$c\in\ZZ_{\geq 0}^{n+3}$ and $k\in\{0,\ldots,n+2\}$, write
\[
v_k(c):=(d_1-1)e_k+e_{k+1}+c,
\]
the exponent of the monomial $x_k^{d_1-1}x_{k+1}x^c$.

We first record the following elementary observation. If $x^c$ has degree
$d_2-d_1$, then
\[
v_k(c)\in \mathcal O(K_2)
\quad\Longleftrightarrow\quad
c=(d_2-d_1)e_k.
\]
Indeed, the $k$-th coordinate of $v_k(c)$ is at least $d_1-1\geq 2$, so it
cannot be the coordinate with value $1$.  Hence, if $v_k(c)\in\mathcal O(K_2)$,
the $k$-th coordinate must be $d_2-1$, which forces $c_k=d_2-d_1$, and since $\sum_i c_i=d_2-d_1$, all other coordinates of $c$ are zero.  The converse
is immediate.

Now let
\[
Q=\sum_c q_c x^c\in S^{d_2-d_1}(V^*)\setminus\{0\}.
\]
We endow $\ZZ_{\geq 0}^{n+3}$ with the lexicographic order induced by
$x_0>x_1>\cdots>x_{n+2}$, and let $c^*$ be the lexicographically largest
element of $M(Q)$.

First assume that $c^*\neq (d_2-d_1)e_0$.  By the observation above,
$v_0(c^*)\notin\mathcal O(K_2)$.  We claim that the monomial with exponent
$v_0(c^*)$ occurs in $K_1Q$ with nonzero coefficient.  Indeed, suppose that
$v_k(c)=v_0(c^*)$ for some $c\in M(Q)$ and some $k$.  If $k=0$, then
$c=c^*$.  If $k\neq 0$, then
\[
c
=
c^*+(d_1-1)e_0+e_1-(d_1-1)e_k-e_{k+1}.
\]

Looking at the $0$-th coordinate, we get
\[
c_0=
\begin{cases}
c_0^*+d_1-1, & \text{if } k\neq n+2,\\
c_0^*+d_1-2, & \text{if } k=n+2.
\end{cases}
\]
Since $d_1\geq 3$, in both cases $c_0>c_0^*$, so $c$ is lexicographically
larger than $c^*$, a contradiction.  Therefore the only contribution to the
coefficient of $x^{v_0(c^*)}$ comes from the product
$x_0^{d_1-1}x_1\cdot x^{c^*}$, and this coefficient is $q_{c^*}\neq 0$.
Thus $K_1Q$ has a monomial whose exponent lies outside $\mathcal O(K_2)$.

It remains to consider the case
\[
c^*=(d_2-d_1)e_0.
\]
Then $v_0(c^*)\in\mathcal O(K_2)$, so instead consider
\[
w:=v_2(c^*)=(d_2-d_1)e_0+(d_1-1)e_2+e_3.
\]
This vector has support of size $3$, because $n+3\geq 4$, and hence
$w\notin\mathcal O(K_2)$.  We claim that $x^w$ occurs in $K_1Q$ with nonzero coefficient. Suppose $v_k(c)=w$ for some $c\in M(Q)$. Since the $k$-th coordinate of $v_k(c)$ is at least $d_1-1$, the $k$-th coordinate of $w$ must be at least $d_1-1$, so $k\in\{0,2\}$ (with $k=0$ possible only when $d_2\geq 2d_1-1$). If $k=0$, then the $1$-st coordinate of $v_0(c)$ is at least $1$, while $w_1=0$, contradicting $v_0(c)=w$. Hence $k=2$, and then necessarily 
\[
c=w-(d_1-1)e_2-e_3=(d_2-d_1)e_0=c^*.
\]
Thus the only contribution to the coefficient of $x^w$ comes from
$x_2^{d_1-1}x_3\cdot x^{c^*}$, and this coefficient is $q_{c^*}\neq 0$.

In all cases, $K_1Q$ has a monomial whose exponent lies outside
$\mathcal O(K_2)$.  Therefore the pair $(K_1,K_2)$ is rigid.
\end{proof}

For $d_1 \geq 4$, a straightforward computation yields $\spar(K_1) = 2(d_1 - 1) \geq 6 > 4$,
and shows that $(\beta^*, \leq_{K_1})$ is the trivial poset whenever $n + 3 \geq 4$.  Both $V(K_1)$ and $V(K_2)$ are smooth by
\cite[Lem.~3.1]{GL13}, and the pair $(K_1, K_2)$ is rigid by
\cref{lem:klein-rigidity}. Hence \cref{thm:CI-main} applies to every Klein complete intersection of
multidegree $(d_1,d_2)$ with $4 \le d_1 < d_2$ and $(n,d_j)\neq(1,4)$
for $j=1,2$. Consequently,
\[
\Aut(X)=\Aut(V(K_1))\cap\Aut(V(K_2))
\]
as subgroups of $\PGL(V)$. Furthermore, if $X$ is smooth, then
\cref{prop:order} implies that, for each $j=1,2$,
\[
(1-d_j)^{\ell_j}\equiv 1 \pmod q
\]
for some $\ell_j\in\{1,\ldots,n+3\}$.  

When the multidegree $(d_1, d_2)$ satisfies the congruence of \cref{prop:reverse}, a sharper conclusion is available: the automorphism group of $X$ coincides with the full automorphism group of the lower-degree Klein hypersurface, and is therefore described explicitly by \cref{prop:aut-klein}. The proof does not rely on \cref{thm:CI-main} and goes directly through the inclusion $\Aut(V(K_1)) \subseteq \Aut(V(K_2))$ given by \cref{prop:reverse}, which allows us to relax the lower bound on $d_1$ to $d_1 \geq 3$.

\begin{theorem} 
\label{thm:klein-CI-wagstaff}
Let $X = V(K_1) \cap V(K_2) \subset \PP^{n+2}$ be a Klein complete intersection of dimension $n \geq 2$ and multidegree $(d_1, d_2)$ with $3 \leq d_1 < d_2$ and $(n, d_1) \neq (2, 3)$. Assume that
\[
(d_1 - 1)(d_2 - 1) \equiv 1 \pmod{m},
\qquad
m \;=\; \frac{(d_1 - 1)^{n+3} - (-1)^{n+3}}{d_1}.
\]
Then
\[
\Aut(X) \;\cong\; \Aut(V(K_1)) \;\cong\; (\ZZ/m\ZZ) \rtimes (\ZZ/(n+3)\ZZ).
\]
\end{theorem}

\begin{proof}
By \cref{lem:F1-preserved}, every $\varphi \in \Aut(X)$ preserves $V(K_1)$, so $\Aut(X) \subseteq \Aut(V(K_1))$. The hypotheses on $(n, d_1)$ ensure that \cref{prop:aut-klein} applies to $V(K_1)$, giving $\Aut(V(K_1)) = G$ with $G = (\ZZ/m\ZZ) \rtimes (\ZZ/(n+3)\ZZ)$. By \cref{prop:reverse}, the congruence $(d_1 - 1)(d_2 - 1) \equiv 1 \pmod{m}$ gives $G \subseteq \Aut(V(K_2))$. Hence every automorphism of $V(K_1)$ preserves the intersection $X = V(K_1) \cap V(K_2)$, so $\Aut(V(K_1)) \subseteq \Aut(X)$. Thus $\Aut(X) = \Aut(V(K_1))$, and the stated group isomorphism follows from \cref{prop:aut-klein}.
\end{proof}

In contrast to the maximal case described in \cref{thm:klein-CI-wagstaff}, the automorphism group of $X$ attains its minimal possible size when $m_1$ and $m_2$ are coprime, in which case the diagonal subgroup is trivial.

\begin{corollary}
\label{cor:minimal-automorphism-group}
Let $X = V(K_1) \cap V(K_2) \subset \PP^{n+2}$ be a Klein complete intersection of dimension $n \geq 1$ and multidegree $(d_1, d_2)$ with $4 \leq d_1 < d_2$ and $(n, d_j) \neq (1, 4)$ for $j = 1, 2$. Assume either $n \geq 2$, or $n = 1$ and $X$ smooth. For each $j \in \{1, 2\}$, let
\[
m_j = \frac{(d_j - 1)^{n+3} - (-1)^{n+3}}{d_j}.
\]
If $\gcd(m_1, m_2) = 1$, then
\[
\Aut(X) \cong \ZZ/(n+3)\ZZ.
\]
\end{corollary}

\begin{proof}
By $d_1 \geq 4$ and $(n, d_j) \neq (1, 4)$, neither $V(K_1)$ nor $V(K_2)$ falls into the exceptional cases of \cref{prop:aut-klein}. Since $V(K_2)$ is isomorphic to a standard Klein hypersurface, \cref{prop:aut-klein} applies to both components, yielding $|\Aut(V(K_j))| = m_j(n+3)$ for $j = 1, 2$. Moreover, $\spar(K_1) = 2(d_1 - 1) \geq 6 > 4$, the poset $(\beta^*, \leq_{K_1})$ is trivial, and the pair $(K_1, K_2)$ is rigid by \cref{lem:klein-rigidity}. Thus \cref{thm:CI-main} applies and gives $\Aut(X) = \Aut(V(K_1)) \cap \Aut(V(K_2))$. By Lagrange's theorem, the order $L = |\Aut(X)|$ divides both group orders, so
\[
L \mid \gcd\big(m_1(n+3),\, m_2(n+3)\big) = (n+3)\gcd(m_1, m_2).
\]
Since $\gcd(m_1, m_2) = 1$, we deduce that $L$ divides $n+3$. On the other hand, the cyclic permutation $\nu$ of order $n+3$ preserves both $K_1$ and $K_2$, so $\Aut(X)$ contains a cyclic subgroup of order $n+3$, whence $L \geq n+3$. Therefore $L = n+3$ and $\Aut(X) \cong \ZZ/(n+3)\ZZ$.
\end{proof}

\begin{example}
\label{ex:klein-CI-extremes}
To illustrate the two geometric extremes governing the automorphism groups of Klein complete intersections, let $n = 2$ and $d_1 = 4$, so that $V(K_1) \subset \PP^4$ is the Klein quartic threefold.  Then
\[
m_1 = \frac{(4-1)^5 - (-1)^5}{4} = 61,
\]
a generalized Wagstaff prime of base $3$.  To observe the maximal automorphism group, we seek a degree $d_2$ satisfying the congruence $(d_1 - 1)(d_2 - 1) \equiv 1 \pmod{m_1}$ of \cref{thm:klein-CI-wagstaff}.  The inverse of $3$ modulo $61$ is $41$, yielding $d_2 = 42$.  Let $X = V(K_1) \cap V(K_2) \subset \PP^4$ be the corresponding Klein complete intersection of multidegree $(4, 42)$.  The hypotheses of \cref{thm:klein-CI-wagstaff} hold.  Consequently, the automorphism group is maximal:
\[
\Aut(X) \cong (\ZZ/61\ZZ) \rtimes (\ZZ/5\ZZ),
\]
a group of order $305$.  In contrast, keeping $V(K_1)$ fixed, consider the intersection with the reverse-order Klein polynomial of degree $d_2 = 5$.  For the multidegree $(4, 5)$,
\[
m_2 = \frac{(5-1)^5 - (-1)^5}{5} = 205,
\]
and $\gcd(m_1, m_2) = \gcd(61, 205) = 1$.  The conditions \cref{cor:minimal-automorphism-group} are again straightforward to check, and so, the diagonal subgroup is entirely obstructed, and the automorphism group collapses to its minimal bound:
\[
\Aut(X) \cong \ZZ/5\ZZ.
\]
\end{example}

\begin{remark}
\label{rmk:main-thm-coverage}
The case $(n, d_1) = (2, 3)$ falls outside both \cref{thm:klein-CI-wagstaff} and \cref{thm:CI-main}. On one hand, $\spar(K_1) = 2(d_1 - 1) = 4$ is not greater than $4$, so \cref{thm:CI-main} does not apply. On the other hand, the Klein cubic threefold $V(K_1) \subset \PP^4$ has automorphism group $\operatorname{PSL}_2(\mathbf{F}_{11})$ \cite{Adl78}, which strictly contains the semidirect product $(\ZZ/11\ZZ) \rtimes (\ZZ/5\ZZ)$ produced by the diagonal and cyclic generators of \cref{prop:reverse}, so the inclusion $\Aut(V(K_1)) \subseteq \Aut(V(K_2))$ cannot be obtained by the present methods, and $\operatorname{PSL}_2(\mathbf{F}_{11})$ is not contained in $\PGP(V, \beta)$ either.  Whether $\Aut(X) = \Aut(V(K_1))$ holds in this case remains open.
\end{remark}

\bibliographystyle{alpha}
\bibliography{nfolds}

\end{document}